\documentclass{aims-ppn}

\usepackage{amsfonts,amsthm,amssymb,amsxtra,color,graphicx,latexsym,paralist}
\usepackage[colorlinks=true]{hyperref}
\hypersetup{urlcolor=blue, citecolor=blue, linkcolor=blue}
\copyrightinfo{2019}{by the authors. This paper may be reproduced, in its entirety, for non-commercial purposes.\\\today}

\def\currentvolume{X}

\def\ppages{X--XX}

\newtheorem{theorem}{Theorem}[section]
\newtheorem{corollary}{Corollary}

\newtheorem{lemma}[theorem]{Lemma}
\newtheorem{proposition}{Proposition}

\theoremstyle{definition}

\newcommand{\R}{{\mathbb R}}
\renewcommand{\S}{{\mathbb S}}

\newcommand{\be}[1]{\begin{equation}\label{#1}}
\newcommand{\ee}{\end{equation}}
\renewcommand{\(}{\left(}
\renewcommand{\)}{\right)}

\newcommand{\iS}[1]{\int_{\S^1}{#1}\,d\sigma}
\newcommand{\nrmS}[2]{\|{#1}\|_{\mathrm L^{#2}(\S^1)}}
\newcommand{\nrmT}[2]{\|{#1}\|_{\mathrm L^{#2}(0,T)}}

\renewcommand{\S}{\mathbb{S}}

\newcommand{\Lp}{\mathcal L_p\kern1pt}
\newcommand{\ez}{z}

\definecolor{darkgreen}{rgb}{0.2,0.7,0.1}

\title[Interpolation, $p$-Laplacian and rigidity]{Interpolation inequalities in $\mathrm W^{1,p}(\S^1)$\\
and \emph{carr\'e du champ} methods}

\author[J.~Dolbeault, M.~Garc\'ia-Huidobro, and R.~Man\'asevich]{}



\email{dolbeaul@ceremade.dauphine.fr}
\email{mgarcia@mat.puc.cl}
\email{manasevi@dim.uchile.cl}

\thanks{$^*$ Corresponding author: Jean Dolbeault}

\begin{document}

\maketitle
\date{\today}
\thispagestyle{empty}

\medskip\centerline{\scshape Jean Dolbeault$\,^*$}
\smallskip{\footnotesize
\centerline{CEREMADE (CNRS UMR n$^\circ$ 7534)}
\centerline{PSL university, Universit\'e Paris-Dauphine}
\centerline{Place de Lattre de Tassigny, 75775 Paris 16, France}}

\medskip\centerline{\scshape Marta Garc\'{i}a-Huidobro}
\smallskip{\footnotesize
\centerline{Departamento de Matem\'{a}ticas}
\centerline{Pontificia Universidad Cat\'{o}lica de Chile}
\centerline{Casilla 306, Correo 22, Santiago de Chile, Chile}}

\medskip\centerline{\scshape R\'aul Man\'asevich}
\smallskip{\footnotesize
\centerline{DIM \& CMM (UMR CNRS n$^\circ$ 2071)}
\centerline{FCFM, Universidad de Chile}
\centerline{Casilla 170 Correo 3, Santiago, Chile}}
\bigskip

\begin{abstract} This paper is devoted to an extension of \emph{rigidity results} for nonlinear differential equations, based on \emph{carr\'e du champ} methods, in the one-dimensional periodic case. The main result is an interpolation inequality with non-trivial explicit estimates of the constants in $\mathrm W^{1,p}(\S^1)$ with $p\ge2$. Mostly for numerical reasons, we relate our estimates with issues concerning periodic dynamical systems. Our interpolation inequalities have a dual formulation in terms of generalized spectral estimates of Keller-Lieb-Thirring type, where the differential operator is now a $p$-Laplacian type operator. It is remarkable that the \emph{carr\'e du champ} method adapts to such a nonlinear framework, but significant changes have to be done and, for instance, the underlying parabolic equation has a nonlocal term whenever $p\neq2$. \end{abstract}

\section{Introduction}\label{Sec:Intro}

This paper is a generalization to $p\neq2$ of results which have been established in~\cite{MR1134481} in the case $p=2$ and go back to~\cite{MR615628}. On the other hand, we use a flow interpretation which was developed in~\cite{DEKL} and relies on the \emph{carr\'e du champ method}. This second approach gives similar results and can be traced back to~\cite{MR889476,MR772092}. As far as we know, Bakry-Emery techniques have been used in the context of the $p$-Laplacian operator to produce estimates of the first eigenvalue but neither for non-linear interpolation inequalities nor for estimates on non-linear $p$-Laplacian flows. By mixing the two approaches, we are able not only to establish inequalities with accurate estimates of the constants but we also obtain \emph{improved inequalities} and get \emph{quantitative rates of convergence} for a nonlinear semigroup associated with the $p$-Laplacian. We also establish \emph{improved rates of convergence}, at least as long as the solution does not enter the asymptotic regime. Results are similar to those of~\cite{Dolbeault_2017} in the case $p=2$.

Let us denote by $\S^1$ the unit circle which is identified with $[0,2\pi)$, with periodic boundary conditions and by $d\sigma=\frac{dx}{2\pi}$ the uniform probability measure on $\S^1$. We define
\[
\lambda_1^\star:=\inf_{v\in\mathcal W_1}\frac{\nrmS{v'}p^2}{\nrmS v2^2}\quad\mbox{and}\quad\lambda_1:=\inf_{v\in\mathcal W_1}\frac{\nrmS{v'}p^2}{\nrmS vp^2}
\]
where the infimum is taken on the set $\mathcal W_1$ of all functions $v$ in $\mathrm W^{1,p}(\S^1)\setminus\{0\}$ such that $\iS v=0$. Here we use the notation:
\[
\nrmS up:=\(\iS{|u|^p}\)^{1/p}\,.
\]
With our notations, $\lambda_1^{p/2}$ is the lowest positive eigenvalue of the $p$-Laplacian operator $\Lp$ defined by
\[
-\,\Lp v:=-\(|v'|^{p-2}\,v'\)'\,.
\]
Since $d\sigma$ is a probability measure, then $\nrmS up^2-\nrmS u2^2$ has the same sign as $(p-2)$, so that
\[
(p-2)\,\big(\lambda_1^\star-\lambda_1\big)\ge0\,.
\]
See Appendix~\ref{Sec:Appendix:Inequalities} for further considerations. Our main result goes as follows.
\begin{theorem}\label{Thm:Main1} Assume that $p\in(2,+\infty)$ and $q>p-1$. There exists $\Lambda_{p,q}>0$ such that for any function $u\in\mathrm W^{1,p}(\S^1)$, the following inequalities hold:
\be{Interp}
\nrmS{u'}p^2\ge\frac{\Lambda_{p,q}}{p-q}\(\nrmS up^2-\nrmS uq^2\)
\ee
if $p\neq q$, and
\be{LogSob}
\nrmS{u'}p^2\ge\frac2p\,\Lambda_{p,p}\,\nrmS up^{2-p}\iS{|u|^p\,\log\(\frac{|u|}{\nrmS up}\)}
\ee
if $p=q$ and $u\not\equiv0$. Moreover, the sharp constant $\Lambda_{p,q}$ in~\eqref{Interp} and~\eqref{LogSob} is such that
\[
\lambda_1\le\Lambda_{p,q}\le\lambda_1^\star\,.
\]
\end{theorem}
Inequality~\eqref{LogSob} is an $\mathrm L^p$ logarithmic Sobolev inequality which is reminiscent of, for instance, \cite{MR1957678}. A Taylor expansion that will be detailed in the proof of Proposition~\ref{Prop:S1} (also see Proposition~\ref{Prop:S1b} and Section~\ref{Sec:InterpIneq}) shows that~\eqref{Interp} and~\eqref{LogSob} tested with $u=1+\varepsilon\,v$ and $v\in\mathcal W_1$ are equivalent at order $\varepsilon^2$ to $\nrmS{v'}p^2\ge\lambda_1^\star\,\nrmS v2^2$, which would not be true if, for instance, we were considering $\nrmS{u'}p^\alpha$ with $\alpha\neq2$. This explains why we have to consider the square of the norms in the inequalities and not other powers, for instance $\alpha=1$ or $\alpha=p$. This is also the reason why $\lambda_1$ and $\lambda_1^\star$ are not defined as the standard first positive eigenvalue of the $p$-Laplacian operator.

In 1992, L.~V\'eron considered in~\cite{MR1151712} the equation
\[
-\,\Lp u+|u|^{q-2}\,u=\lambda\,|u|^{p-2}\,u
\]
not only on $\S^1$ but also on general manifolds and proved that it has no solution in $\mathrm W^{1,p}(\S^1)$ except the constant functions if $\lambda\le\lambda_1^{p/2}$ and $1<p<q$. Let us point out that, up to constants that come from the various norms involved in~\eqref{Interp}, the corresponding Euler-Lagrange equation is the same equation when $q<p$, while in the case $q>p$, the Euler-Lagrange equation is, again up to constants that involve the norms, of the form
\[
-\,\Lp u+\lambda\,|u|^{p-2}\,u=|u|^{q-2}\,u\,.
\]

This paper is organized as follows. In Section~\ref{Sec:proof}, we start by proving Theorem~\ref{Thm:Main1} in the case $2<p<q$. The key estimate is the Poincar\'e type estimate of Lemma~\ref{Lem:CS}, which is used in Section~\ref{Sec:rigidity} to prove Proposition~\ref{Prop:M1}. The adaptations needed to deal with the case $q<p$ are listed in Section~\ref{Sec:Extension}.

Section~\ref{Sec:Further} is devoted to further results and consequences. In Section~\ref{Sec:Parabolic}, we give an alternative proof of Theorem~\ref{Thm:Main1} based on a parabolic method. This is the link to the  \emph{carr\'e du champ} methods. The parabolic setting provides a framework in which the computations of Section~\ref{Sec:proof} can be better interpreted. Another consequence of the parabolic approach is that a refined estimate is established by taking into account terms that are simply dropped in the elliptic estimates of Section~\ref{Sec:proof}: see Section~\ref{Sec:improvement}. A last result deals with ground state energy estimates for nonlinear Schr\"odinger type operators, which  generalize to the case of the $p$-Laplacian the Keller-Lieb-Thirring estimates known when $p=2$: see Section~\ref{Sec:Keller-Lieb-Thirring}. Notice that such estimates are completely equivalent to the interpolation inequalities of  Theorem~\ref{Thm:Main1}, including for optimality results.

Numerical results which illustrate our main theoretical results have been collected in Section~\ref{Sec:numerics}. The computations are relatively straightforward because, after a rescaling, the bifurcation problem (described below in Sections~\ref{Sec:Var} and \ref{Sec:Extension}) can be rephrased as a dynamical system such that all quantities associated with critical points can be computed in terms of explicit integrals. Various technical results are collected in the Appendix.

\section{Proof of the main result}\label{Sec:proof}

The goal of this section is to prove Theorem~\ref{Thm:Main1} and some additional results. The emphasis is put on the case $2<p<q$, while the other cases are only sketched.

\subsection{A variational problem}\label{Sec:Var}

On $\S^1$, let us assume that $p<q$ and define
\[
\mathcal Q_\lambda[u]:=\frac{\nrmS{u'}p^2+\lambda\,\nrmS up^2}{\nrmS uq^2}
\]
for any $\lambda>0$. Let
\[
\mu(\lambda):=\inf_{u\in\mathrm W^{1,p}(\S^1)\setminus\{0\}}\mathcal Q_\lambda[u]\,.
\]
In the range $p<q$, inequality~\eqref{Interp} can be embedded in the larger family of inequalities
\be{Interp-1}
\nrmS{u'}p^2+\lambda\,\nrmS up^2\ge\mu(\lambda)\,\nrmS uq^2\quad\forall\,u\in\mathrm W^{1,p}(\S^1)
\ee
so that the optimal constant $\Lambda_{p,q}$ of Theorem~\ref{Thm:Main1} can be characterized as
\[
\Lambda_{p,q}=(q-p)\inf\big\{\lambda>0\,:\,\mu(\lambda)<\lambda\big\}\,.
\]
\begin{proposition}\label{Prop:S1} Assume that $1<p<q$. On $(0,+\infty)$, the function $\lambda\mapsto\mu(\lambda)$ is concave, strictly increasing, such that $\mu(\lambda)<\lambda$ if $\lambda>\lambda_1^\star/(q-p)$.\end{proposition}
\begin{proof} The concavity is a consequence of the definition of $\mu(\lambda)$ as an infimum of affine functions of $\lambda$. If $\mu(\lambda)=\lambda$, then the equality is achieved by constant functions. If we take
\[
u=1+\varepsilon\,v\quad\mbox{with}\quad\iS v=0
\]
as a test function for $\mathcal Q_\lambda$ and let $\varepsilon\to0$, then we get
\[
\mathcal Q_\lambda[1+\varepsilon\,v]-\lambda\sim\varepsilon^2\(\nrmS{v'}p^2-\lambda\,(q-p)\,\nrmS v2^2\)\,.
\]
Let us take an optimal $v$ for the minimization problem corresponding to $\lambda_1^\star$, so that the r.h.s.~becomes proportional to $\lambda_1^\star-\lambda\,(q-p)$. As a consequence, we know that $\mu(\lambda)<\lambda$ if $\lambda>\lambda_1^\star/(q-p)$.
\end{proof}

By standard methods of the calculus of variations, we know that the infimum $\mu(\lambda)$ is achieved for any $\lambda>0$ by some a.e.~positive function in $W^{1,p}(\S^1)$. As a consequence, we know that there are multiple solutions to the Euler-Lagrange equation if $\lambda>\lambda_1^\star/(q-p)$. This equation can be written as
\be{ELS0}
-\,\nrmS{u'}p^{2-p}\,\Lp u+\lambda\,\nrmS up^{2-p}\,u^{p-1}=\mu(\lambda)\,\nrmS uq^{2-q}\,u^{q-1}\,.
\ee
What we want to prove is that~\eqref{ELS0} has no non-constant solution for $\lambda>0$, small enough.
\begin{proposition}\label{Prop:M1} Assume that $2<p<q$ and $\lambda>0$. Equation~\eqref{ELS0} has a unique positive solution, up to a multiplication by a constant, if $\lambda\le\lambda_1/(q-p)$. This means that only constants are solutions.\end{proposition}
The proof of this result is given in Section~\ref{Sec:rigidity}. As a preliminary step, we establish a Poincar\'e estimate.

\subsection{A Poincar\'e estimate}\label{Sec:Poincare}

Let us consider the Poincar\'e inequality
\[
\nrmS{v'}p^2-\lambda_1\,\nrmS vp^2\ge0\quad\forall\,v\in\mathcal W_1\,,
\]
which is a consequence of the definition of $\lambda_1$.
\begin{lemma}\label{Lem:CS} Assume that $p>2$. Then for any $u\in\mathrm W^{2,p}(\S^1)\setminus\{0\}$, we have
\be{Ineq:2.3}
\iS{u^{2-p}\,(\Lp u)^2}\ge\lambda_1\,\frac{\nrmS{u'}p^{2(p-1)}}{\nrmS up^{p-2}}\,.
\ee
Moreover $\lambda_1$ is the sharp constant.
\end{lemma}
\begin{proof} By expanding the square, we know that
\begin{multline*}
0\le\iS{u^{2-p}\,\left|\Lp u+C\,|u|^{p-2}\,(u-\bar u)\right|^2}=\iS{u^{2-p}\,(\Lp u)^2}-\,C\,\nrmS{u'}p^p\\
-\,C\(\nrmS{u'}p^p-\,C\iS{|u|^{p-2}\,(u-\bar u)^2}\)\,.
\end{multline*}
With
\[
C=\lambda_1\,\frac{\nrmS{u'}p^{p-2}}{\nrmS up^{p-2}}\,,\quad v=u-\bar u\quad\mbox{and}\quad\bar u=\iS u\,,
\]
we know that
\begin{multline*}
\iS{\frac{u^{2-p}}{\nrmS up^{2-p}}\,(\Lp u)^2}-\lambda_1\,\nrmS{u'}p^{2(p-1)}\\
\ge\lambda_1\,\nrmS{u'}p^{2(p-2)}\(\nrmS{u'}p^2-\lambda_1\iS{\frac{|u|^{p-2}}{\nrmS up^{p-2}}\,v^2}\)\,.
\end{multline*}
Assuming that $p>2$, H\"older's inequality with exponents $p/(p-2)$ and $p/2$ shows that
\[
\iS{\frac{|u|^{p-2}}{\nrmS up^{p-2}}\,v^2}\le\nrmS up^{2-p}\left[\iS{\(|u|^{p-2}\)^\frac p{p-2}}\right]^\frac{p-2}p\,\nrmS vp^2=\nrmS vp^2\,.
\]
We observe that $v$ is in $\mathcal W_1$ so that $\lambda_1\,\nrmS vp^2\le\nrmS{v'}p^2$ and
\[
\lambda_1\iS{|u|^{p-2}\,v^2}\le\nrmS{u'}p^2\,\nrmS up^{p-2}
\]
(see~\eqref{Ineq4} for further considerations). Hence we conclude that
\begin{multline*}
\iS{\frac{u^{2-p}}{\nrmS up^{2-p}}\,(\Lp u)^2}-\lambda_1\,\nrmS{u'}p^{2(p-1)}\\
\ge\lambda_1\,\nrmS{u'}p^{2(p-2)}\(\nrmS{v'}p^2-\lambda_1\,\nrmS vp^2\)\ge0\,.
\end{multline*}
The fact that $\lambda_1$ is optimal is obtained by considering the equality case in the above inequalities. See details in Appendix~\ref{Sec:Appendix:Inequalities}.\end{proof}

\subsection{A first rigidity result}\label{Sec:rigidity}

We adapt the strategy of~\cite{MR1134481,MR615628} when $p=2$ to the case $p>2$ using the Poincar\'e estimate of Section~\ref{Sec:Poincare}.
\begin{proof}[Proof of Proposition~\ref{Prop:M1}]
Let us consider a positive solution to~\eqref{ELS0}. If we multiply~\eqref{ELS0} by $-\,u^{2-p}\Lp u$ and integrate on~$\S^1$, we obtain the identity
\begin{multline*}
\nrmS{u'}p^{2-p}\iS{u^{2-p}\,(\Lp u)^2}+\lambda\,\nrmS up^{2-p}\,\iS{|u'|^p}\\
=(1+q-p)\,\mu(\lambda)\,\nrmS uq^{2-q}\iS{u^{q-p}\,|u'|^p}\,.
\end{multline*}
If we multiply~\eqref{ELS0} by $(1+q-p)\,u^{1-p}\,|u'|^p$ and integrate on $\S^1$, we obtain the identity
\begin{multline*}
-\,\nrmS{u'}p^{2-p}\,(1+q-p)\iS{u^{1-p}\Lp u\,|u'|^p}+\lambda\,\nrmS up^{2-p}\,(1+q-p)\iS{|u'|^p}\\
=(1+q-p)\,\mu(\lambda)\,\nrmS uq^{2-q}\iS{u^{q-p}\,|u'|^p}\,.
\end{multline*}
By subtracting the second identity from the first one, we obtain that
\begin{multline*}
\nrmS{u'}p^{2-p}\(\iS{u^{2-p}\,(\Lp u)^2}+\,(1+q-p)\iS{u^{1-p}\Lp u\,|u'|^p}\)\\
-\lambda\,(q-p)\,\nrmS up^{2-p}\iS{|u'|^p}=0\,.
\end{multline*}
After an integration by parts, the above identity can be rewritten as
\begin{multline*}\label{Eqn:Identity1}
\iS{u^{2-p}\,(\Lp u)^2}+\,\frac{(1+q-p)\,(p-1)^2}{2\,p-1}\iS{\frac{|u'|^{2p}}{u^p}}\\
-\lambda\,(q-p)\,\frac{\nrmS{u'}p^{p-2}}{\nrmS up^{p-2}}\iS{|u'|^p}=0\,.
\end{multline*}
By Lemma~\ref{Lem:CS}, this proves that
\[
\big(\lambda_1-\lambda\,(q-p)\big)\,\frac{\nrmS{u'}p^{2(p-1)}}{\nrmS up^{p-2}}+\,\frac{(1+q-p)\,(p-1)^2}{2\,p-1}\iS{\frac{|u'|^{2p}}{u^p}}\le0\,.
\]
If $\lambda\le\lambda_1/(q-p)$, this proves that $u$ is a constant. This completes the proof of Proposition~\ref{Prop:M1}.\end{proof}

\subsection{An extension of the range of the parameters}\label{Sec:Extension}

So far we have considered only the case $q>p$. Let us consider the case $1<q<p$ and define, in that case, $\mathcal Q^\mu$ by
\[
\mathcal Q^\mu[u]:=\frac{\nrmS{u'}p^2+\mu\,\nrmS uq^2}{\nrmS up^2}
\]
for any $\mu>0$. Let
\[
\lambda(\mu):=\inf_{u\in\mathrm W^{1,p}(\S^1)\setminus\{0\}}\mathcal Q^\mu[u]\,.
\]
If $\mu(\lambda)=\lambda$, then the equality is achieved by constant functions. In the range $p>q$, inequality~\eqref{Interp} can be embedded in the larger family of inequalities
\be{Interp-2}
\nrmS{u'}p^2+\mu\,\nrmS uq^2\ge\lambda(\mu)\,\nrmS up^2\quad\forall\,u\in\mathrm W^{1,p}(\S^1)
\ee
so that the optimal constant $\Lambda_{p,q}$ of Theorem~\ref{Thm:Main1} can be characterized as
\[
\Lambda_{p,q}=(p-q)\,\inf\big\{\mu>0\,:\,\lambda(\mu)<\mu\big\}\,.
\]
As in Section~\ref{Sec:Var}, a Taylor expansion allows us to prove the following result.
\begin{proposition}\label{Prop:S1b} Assume that $1<q<p$. The function $\mu\mapsto\lambda(\mu)$ is concave, strictly increasing, such that $\lambda(\mu)<\mu$ if $\mu>\lambda_1^\star/(p-q)$.\end{proposition}
As a consequence of Proposition~\ref{Prop:S1b}, there are non-constant solutions to the Euler-Lagrange equation if $\mu>\lambda_1^\star/(p-q)$ which means multiple solutions. This equation can be written as
\be{ELS0b}
-\,\nrmS{u'}p^{2-p}\,\Lp u+\mu\,\nrmS uq^{2-q}\,u^{q-1}=\lambda(\mu)\,\nrmS up^{2-p}\,u^{p-1}\,.
\ee
There is also a uniqueness range.
\begin{proposition}\label{Prop:M2} Assume that $p>2$ and $p-1<q<p$. Equation~\eqref{ELS0b} has a unique positive solution, up to a multiplication by a constant, if $\mu\le\lambda_1/(p-q)$.\end{proposition}
\begin{proof} The computation is exactly the same as in the proof of Proposition~\ref{Prop:M1}, except that $\lambda$ and $\mu(\lambda)$ have to be replaced by $-\lambda(\mu)$ and $-\mu$ respectively.
\begin{multline*}\label{Eqn:Identity1}
\iS{u^{2-p}\,(\Lp u)^2}+\,\frac{(1+q-p)\,(p-1)^2}{2\,p-1}\iS{\frac{|u'|^{2p}}{u^p}}\\
-\lambda\,(p-q)\,\frac{\nrmS{u'}p^{p-2}}{\nrmS up^{p-2}}\iS{|u'|^p}=0\,.
\end{multline*}
The conclusion then holds by Lemma~\ref{Lem:CS} as in the case $q>p$.\end{proof}

\subsection{The interpolation inequalities}\label{Sec:InterpIneq}

\begin{proof}[Proof of Theorem~\ref{Thm:Main1}] Inequality~\eqref{Interp} follows from Inequalities~\eqref{Interp-1} and~\eqref{Interp-2} with
\begin{align*}
\Lambda_{p,q}:=&(q-p)\,\min\big\{\lambda>0\,:\,\mu(\lambda)<\lambda\big\}\quad\mbox{if}\quad2<p<q\,,\\
\Lambda_{p,q}:=&(p-q)\,\min\big\{\mu>0\,:\,\lambda(\mu)<\mu\big\}\quad\mbox{if}\quad p>2\quad\mbox{and}\quad p-1<q<p\,.
\end{align*}
It remains to consider the limit case as $q\to p$. By passing to the limit in the right hand side, we obtain the $\mathrm L^p$ logarithmic Sobolev inequality~\eqref{LogSob}. The upper bound $\Lambda_{p,p}\le\lambda_1^\star$ is easily checked by computing
\begin{multline*}
\nrmS{u_\varepsilon'}p^2-\frac2p\,\lambda\,\nrmS{u_\varepsilon}p^{2-p}\iS{|u_\varepsilon|^p\,\log\(\frac{|u_\varepsilon|}{\nrmS{u_\varepsilon}p}\)}\\
=\varepsilon^2\,(\lambda_1-\lambda)\nrmS{v}p^{2-p}+o(\varepsilon^2)
\end{multline*}
where $u_\varepsilon=1+\varepsilon\,v$ and $v$ is an optimal function for the minimization problem corresponding to $\lambda_1^\star$.\end{proof}

\section{Further results and consequences}\label{Sec:Further}

In this section we collect a list of results which go beyond the statement of Theorem~\ref{Thm:Main1}. Let us start with an alternative proof of this results which paves the route to an improved interpolation inequality, compared to inequality~\eqref{Interp}.

\subsection{The parabolic point of view}\label{Sec:Parabolic}

As in~\cite{DEKL}, the method of Section~\ref{Sec:rigidity} can be rephrased using an evolution problem in the framework of the \emph{carr\'e du champ} method. Here we shall consider the $1$-homogenous $p$-Laplacian flow
\be{flow}
\frac{\partial u}{\partial t}=\frac{\nrmS{u'}p^{2-p}}{\nrmS up^{2-p}}\,u^{2-p}\(\Lp u+(1+q-p)\,\frac{|u'|^p}u\)\,.
\ee
The main originality compared to previous results based on the \emph{carr\'e du champ} method is that a nonlocal term involving the norms $\nrmS{u'}p$ and $\nrmS up$ has to be introduced in order to obtain a linear estimate of the \emph{entropy}, defined as
\[
\mathsf e(t):=\frac{\nrmS up^2-\nrmS uq^2}{p-q}
\]
if $p\neq q$, in terms of the \emph{Fisher information}
\[
\mathsf i(t):=\nrmS{u'}p^2\,.
\]
If $u$ is a positive solution of~\eqref{flow}, we first observe that
\[
\frac d{dt}\iS{u^q}=q\,\frac{\nrmS{u'}p^{2-p}}{\nrmS up^{2-p}}\iS{u^{1+q-p}\(\Lp u+(1+q-p)\,\frac{|u'|^p}u\)}=0\,.
\]
Hence $\nrmS uq^2$ does not depend on $t$ and we may assume without loss of generality that $\nrmS uq^2=1$. After an integration by parts,
\[
\mathsf e'=\frac{d\,\mathsf e}{dt}=2\,\frac{\nrmS{u'}p^{2-p}}{p-q}\iS{u\(\Lp u+(1+q-p)\,\frac{|u'|^p}u\)}=-\,2\,\mathsf i
\]
if $p\neq q$. A similar computation shows that $\mathsf e'=-\,2\,\mathsf i$ is also true when $p=q$, provided we define the entropy by
\[
\mathsf e(t):=\frac2p\,\nrmS up^{2-p}\iS{|u|^p\,\log\(\frac{|u|}{\nrmS up}\)}
\]
in that case, with $\mathsf i:=\nrmS{u'}p^2$ as before.

One more derivation along the flow shows that
\[
\mathsf i'=\frac{d\,\mathsf i}{dt}=-\,2\,\frac{\nrmS{u'}p^{2(2-p)}}{\nrmS up^{2-p}}\iS{\big(\Lp u\big)\,u^{2-p}\(\Lp u+(1+q-p)\,\frac{|u'|^p}u\)}\,.
\]
Using an integration by parts, we have
\[
\iS{\big(\Lp u\big)\,u^{2-p}\,\frac{|u'|^p}u}=\frac{(p-1)^2}{2\,p-1}\iS{\frac{|u'|^{2p}}{u^p}}\ge0\,.
\]
With the help of Lemma~\ref{Lem:CS}, we conclude that
\be{didt}
\mathsf i'\le-\,2\,\lambda_1\,\mathsf i-\,2\,(1+q-p)\,\frac{(p-1)^2}{2\,p-1}\,\frac{\nrmS{u'}p^{2(2-p)}}{\nrmS up^{2-p}}\iS{\frac{|u'|^{2p}}{u^p}}
\le-\,2\,\lambda_1\,\mathsf i\,.
\ee
Hence a positive solution of~\eqref{flow} is such that $\mathsf i(t)\le\mathsf i(0)\,e^{-\,2\,\lambda_1\,t}$ for any $t\ge0$ and thus $\lim_{t\to+\infty}\mathsf i(t)=0$. We will see next that $\lim_{t\to+\infty}\mathsf e(t)=0$. As $\nrmS uq^2=1$, for each $n\in\mathbb N$, there exists $x_n\in[0,2\pi)$ such that $u(x_n,n)=1$, hence
\[
|u(x,n)-1|=|u(x,n)-u(x_n,n)|\le C\,\nrmS{u'(\cdot,n)}p
\]
implying that $|\nrmS{u(\cdot,n)}p-1|\to 0$ as $n\to\infty$. Since $\lim_{t\to\infty}\mathsf e(t)$ exists, our claim follows. After observing that
\[
\frac d{dt}\(\mathsf i-\,\lambda_1\,\mathsf e\)\le0\quad\mbox{and}\quad\lim_{t\to+\infty}\big(\mathsf i(t)-\,\lambda_1\,\mathsf e(t)\big)=0\,,
\]
we conclude that $\mathsf i-\,\lambda_1\,\mathsf e\ge0$ at any $t\ge0$ and, as a special case, at $t=0$, for an arbitrary initial datum. This is already an alternative proof of Theorem~\ref{Thm:Main1}. The inequality $\mathsf i\ge\lambda_1\,\mathsf e$ is usually referred to as the \emph{entropy -- entropy production inequality} in the literature: see for instance~\cite{MR2065020}.

So far, this proof is formal as we did not establish the existence of the solutions to the parabolic problem nor the regularity which is needed to justify all steps. To make the proof rigorous, here are the main steps that have to be done: \begin{enumerate}
\item Regularize the initial datum to make it as smooth as needed and bound it from below by a positive constant, and from above by another positive constant.
\item Regularize the operator by considering for instance the operator
\[
u\mapsto(p-1)\,\(\varepsilon^2+|u'|^2\)^{\frac p2-1}\,u''
\]
for an arbitrarily small $\varepsilon>0$.
\item Prove estimates of the various norms based on the adapted (for $\varepsilon>0$) equation and on entropy estimates as above, and establish that these estimates can be obtained uniformly in the limit as $\varepsilon\to0$.
\item Get inequalities (with degraded constants for $\varepsilon>0$) and recover entropy -- entropy production inequality by taking the limit as $\varepsilon\to0$.
\item Conclude by density on the inital datum, in order to prove the result in the Sobolev space of Theorem~\ref{Thm:Main1}.
\end{enumerate}
Details are out of the scope of the present paper. None of these steps is extremely difficult but lots of care is needed. Regularity and justification of the integrations by parts is a standard issue in this class of problems, see for instance the comments in~\cite[page~694]{MR2459454}.

\subsection{An improvement of the interpolation inequality}\label{Sec:improvement}

The parabolic approach provides an easy improvement of~\eqref{Interp} and~\eqref{LogSob}. Let us consider the function~$\Psi$ defined by
\[
\Psi(\ez)=\ez+\frac{(p-1)^2}{2\,(2\,p-1)}\,(1+q-p)\,\lambda_1^{p-2}\int_0^{\ez}\frac{s^{p-1}}{1+(p-q)\,s}\,ds\,.
\]
The function $\Psi$ is defined on $\R^+$, convex and such that $\Psi(0)=0$ and $\Psi'(0)=1$. The flow approach provides us with an improved version of Theorem~\ref{Thm:Main1}.
\begin{theorem}\label{Thm:Main2} Assume that $p\in(2,+\infty)$ and $q>p-1$. For any function $u\in\mathrm W^{1,p}(\S^1)\setminus\{0\}$, the following inequalities hold:
\[
\nrmS{u'}p^2\ge\lambda_1\,\nrmS uq^2\,\Psi\(\frac1{p-q}\,\frac{\nrmS up^2-\nrmS uq^2}{\nrmS uq^2}\)
\]
if $p\neq q$, and
\[
\nrmS{u'}p^2\ge\lambda_1\,\nrmS up^2\,\Psi\(\frac2p\iS{\frac{|u|^p}{\nrmS up^p}\,\log\(\frac{|u|}{\nrmS up}\)}\)
\]
if $p=q$.\end{theorem}
\begin{proof} In the computations of Section~\ref{Sec:Parabolic},~\eqref{didt}, we dropped the term $\iS{\frac{|u'|^{2p}}{u^p}}$. Actually, the Cauchy-Schwarz inequality
\[
\(\iS{|u'|^p}\)^2=\(\iS{u^\frac p2\cdot u^{-\frac p2}|u'|^p}\)^2\le\iS{u^p}\iS{\frac{|u'|^{2p}}{u^p}}
\]
can be used as in~\cite{MR2152502} to prove that
\[
\iS{\frac{|u'|^{2p}}{u^p}}\ge\frac{\(\iS{|u'|^p}\)^2}{\iS{u^p}}=\frac{\mathsf i^p}{1+(p-q)\,\mathsf e}
\]
with $\mathsf i=\nrmS{u'}p^2$ and, after recalling that $\mathsf e'=-\,2\,\mathsf i$, deduce from~\eqref{didt} that
\[\label{ODI}
\mathsf e''+2\,\lambda_1\,\mathsf e'-\,2\,\kappa\,\frac{|\mathsf e'|^2\,\mathsf i^{p-2}}{1+(p-q)\,\mathsf e}\ge0
\]
with $\kappa=\frac{(p-1)^2}{2\,(2\,p-1)}\,(1+q-p)$. Here we assume that $\nrmS uq=1$ so that $\nrmS up^2=1+(p-q)\,\mathsf e$. Using the standard entropy -- entropy production inequality $\mathsf i-\,\lambda_1\,\mathsf e\ge0$, we deduce that
\[
\mathsf e''+\,2\,\lambda_1\(1+\,\kappa\,\lambda_1^{p-2}\,\frac{\mathsf e^{p-1}}{1+(p-q)\,\mathsf e}\)\mathsf e'\ge0\,,
\]
that is,
\[
\frac d{dt}\(\mathsf i-\,\lambda_1\,\Psi(\mathsf e)\)\le0\,,
\]
and the result again follows from $\lim_{t\to+\infty}\left[\mathsf i(t)-\,\lambda_1\,\Psi\big(\mathsf e(t)\big)\right]=0$ if $\nrmS uq=1$. The general case is obtained by replacing $u$ by $u/\nrmS uq=1$.
\end{proof}

We notice that the equality in~\eqref{Interp} can be achieved only by constants because
\[
\mathsf i-\,\lambda_1\,\mathsf e\ge\lambda_1\,\big(\Psi(\mathsf e)-\mathsf e\big)\ge0
\]
and $\Psi(\ez)-\ez=0$ is possible if and only if $\ez=0$. This explains why the infima of $\mathcal Q_{\lambda_1}$ and $\mathcal Q^{\lambda_1}$ are achieved only by constant functions. In the case $p=q$, let us notice that the \emph{improved interpolation inequality} of Theorem~\ref{Thm:Main2} can be written as
\begin{multline*}
\nrmS{u'}p^2\ge\frac{2\,\lambda_1}p\,\nrmS up^{2-p}\iS{|u|^p\,\log\(\frac{|u|}{\nrmS up}\)}\\
+\lambda_1\,\mathsf a\(\frac2p\,\nrmS up^{2-p}\iS{|u|^p\,\log\(\frac{|u|}{\nrmS up}\)}\)^p
\end{multline*}
with $\mathsf a:=\frac{(p-1)^2}{2\,p\,(2\,p-1)}\,(1+q-p)\,\lambda_1^{p-2}$.

In the above computations, we use the estimate $\mathsf i-\,\lambda_1\,\mathsf e\ge0$ to show that $|\mathsf e'|^2\,\mathsf i^{p-2}\ge-\lambda_1^{p-1}\,\mathsf e^{p-1}\,\mathsf e'$. This estimate is crude and there is space for improvement.

\subsection{Keller-Lieb-Thirring estimates}\label{Sec:Keller-Lieb-Thirring}

The nonlinear interpolation inequalities \eqref{Interp-1} and~\eqref{Interp-2} can be used to get estimates of the \emph{ground state energy} of Keller-Lieb-Thirring type as, for instance, in~\cite{DolEsLa-APDE2014}.

\medskip In the range $2<p<q$, by applying H\"older's inequality, we find that
\[
\iS{V\,|u|^p}\le\nrmS V{\frac q{q-p}}\,\nrmS uq^p\,.
\]
We can rewrite~\eqref{Interp-1} as
\[
\nrmS{u'}p^2-\mu\,\nrmS uq^2\ge-\,\lambda\,\nrmS up^2
\]
with $\mu=\nrmS V{\frac q{q-p}}$ and $\lambda=\lambda(\mu)$ computed as the inverse of the function $\lambda\mapsto\lambda(\mu)$, according to Proposition~\ref{Prop:S1}. As a consequence of Propositions~\ref{Prop:S1} and~\ref{Prop:M1}, we have the following estimate on the \emph{ground state energy}.
\begin{corollary}\label{Cor:KLT1} Assume that $2<p<q$. With the above notations, for any function $V\in\mathrm L^{\frac q{q-p}}(\S^1)$ and any $u\in\mathrm W^{1,p}(\S^1)$, we have the estimate
\[
\nrmS{u'}p^2-\iS{V\,|u|^p}\ge-\,\lambda\(\nrmS V{\frac q{q-p}}\)\,\nrmS up^2\,.
\]
Moreover,
\[
\lambda\(\nrmS V{\frac q{q-p}}\)=\nrmS V{\frac q{q-p}}\quad\mbox{if}\quad\nrmS V{\frac q{q-p}}\le\frac{\lambda_1}{q-p}
\]
and in that case, the equality is realized if and only if $V$ is constant.
\end{corollary}

If $p>2$ and $p-1<q<p$, there is a similar estimate, which goes as follows. By applying H\"older's inequality, we find that
\[
\iS{|u|^q}=\iS{V^{-\frac qp}\,V^\frac qp\,|u|^q}\le\(\iS{V^{-\frac q{p-q}}}\)^{1-\frac qp}\(\iS{V\,|u|^p}\)^\frac qp
\]
so that
\[
\iS{V\,|u|^p}\ge\nrmS{V^{-1}}{\frac q{p-q}}^{-1}\,\nrmS uq^p\,.
\]
We can rewrite~\eqref{Interp-2} as
\[
\nrmS{u'}p^2+\mu\,\nrmS uq^2\ge\lambda\,\nrmS up^2
\]
with $\mu=\nrmS{V^{-1}}{\frac q{p-q}}^{-1}$ and $\lambda=\lambda(\mu)$, according to Proposition~\ref{Prop:S1b}. As a consequence of Propositions~\ref{Prop:S1b} and~\ref{Prop:M2}, we have the following estimate on the \emph{ground state energy}.
\begin{corollary}\label{Cor:KLT2} Assume that $p>2$ and $p-1<q<p$. With the above notations, for any function $V\in\mathrm L^{\frac q{q-p}}(\S^1)$ and any $u\in\mathrm W^{1,p}(\S^1)$, we have the estimate
\[
\nrmS{u'}p^2+\iS{V\,|u|^p}\ge\lambda\(\nrmS{V^{-1}}{\frac q{p-q}}^{-1}\)\,\nrmS up^2\,.
\]
Moreover,
\[
\lambda\(\nrmS{V^{-1}}{\frac q{p-q}}^{-1}\)=\nrmS{V^{-1}}{\frac q{p-q}}^{-1}\quad\mbox{if}\quad\nrmS{V^{-1}}{\frac q{p-q}}^{-1}\le\frac{\lambda_1}{p-q}
\]
and in that case, the equality is realized if and only if $V$ is constant.
\end{corollary}

\section{Numerical results}\label{Sec:numerics}

Equation~\eqref{ELS0} involves non-local terms, which raises a numerical difficulty. However, using the homogeneity and a scaling, it is possible to formulate an equivalent equation without non-local terms and use it to perform some numerical computations.

\subsection{A reparametrization}

A solution of~\eqref{ELS0} can be seen as $2\pi$-periodic solution on $\R$. By the rescaling
\be{Rescaling}
u(x)=K\,f\(\frac T{2\pi}\,(x-x_0)\)\,,
\ee
we get that $f$ solves
\[
-\(\frac T{2\pi}\)^p\,\nrmS{u'}p^{2-p}\,K^{p-1}\Lp f+\lambda\,\nrmS up^{2-p}\,K^{p-1}f^{p-1}=\mu(\lambda)\,\nrmS uq^{2-q}\,K^{q-1}\,f^{q-1}\,.
\]
We can adjust $T$ so that
\[
\(\frac T{2\pi}\)^p\,\nrmS{u'}p^{2-p}=\lambda\,\nrmS up^{2-p}
\]
and $K$ so that
\[
K^{q-p}\,\mu(\lambda)\,\nrmS uq^{2-q}=\lambda\,\nrmS up^{2-p}\,.
\]
Altogether, this means that the function $f$ now solves
\be{ELS2}
-\,\Lp f+f^{p-1}=f^{q-1}
\ee
on $\R$ and is $T$-periodic. Equations~\eqref{ELS0} and~\eqref{ELS2} are actually equivalent.
\begin{proposition}\label{Prop:Equivalence} Assume that $p\in(2,+\infty)$ and $q>p-1$. If $u>0$ solves~\eqref{ELS0} and $f$ is given by~\eqref{Rescaling} with $x_0\in\R$,
\[
T=2\pi\,\lambda^\frac1p\,\frac{\nrmS{u'}p^{1-\frac2p}}{\nrmS up^{1-\frac2p}}\quad\mbox{and}\quad K=\(\frac\lambda{\mu(\lambda)}\,\frac{\nrmS uq^{q-2}}{\nrmS up^{p-2}}\)^\frac1{q-p}\,,
\]
then $f$ solves~\eqref{ELS2} and it is $T$-periodic. Reciprocally, if $f$ is a $T$-periodic positive solution of~\eqref{ELS2}, then~$u$ given by~\eqref{Rescaling} is, for an arbitrary $x_0\in\R$, and an arbitrary $K>0$, a $2\pi$-periodic positive solution of~\eqref{ELS0} with
\[
\lambda=\(\frac T{2\pi}\)^2\,\frac{\nrmT fp^{p-2}}{\nrmT{f'}p^{p-2}}
\]
and
\[
\mu(\lambda)=\lambda\,T^{\frac2q-\frac2p}\,\frac{\nrmT fq^{q-2}}{\nrmT fp^{p-2}}=\(\frac T{2\pi}\)^2\,T^{\frac2q-\frac2p}\,\frac{\nrmT fq^{q-2}}{\nrmT{f'}p^{p-2}}\,.
\]
\end{proposition}
\begin{proof} To see that $u(x)=f(T\,x\,/(2\pi))$ solves~\eqref{ELS0}, it is enough to write~\eqref{ELS2} in terms of $u$ and use the change of variables to get that
\begin{multline*}
\nrmS{u'}p=\frac1{2\pi}\,T^{1-\frac1p}\,\nrmT{f'}p\,,\quad\nrmS up=T^{-\frac1p}\,\nrmT fp\,,\\
\mbox{and}\quad\nrmS uq=T^{-\frac1q}\,\nrmT fq\,.
\end{multline*}
Notice that on $\S^1$, we use the uniform probability measure $d\sigma$, while on $[0,T]$ we use the standard Lebesgue measure. We can of course translate $u$ by $x_0\in\R$ or multiply it by an arbitrary $K>0$ (or an arbitrary $K\in\R$ if we relax the positivity condition).\end{proof}

Proposition~\ref{Prop:M1} and Proposition~\ref{Prop:Equivalence} have a straightforward consequence on the period of the solutions of~\eqref{ELS2}.
\begin{corollary}\label{Cor:rigidity} Assume that $2<p<q$. If $f$ is a non-constant periodic solution of~\eqref{ELS2} of period $T$, then
\[
T>2\,\pi\,\sqrt{\frac{\lambda_1}{q-p}}\,\frac{\nrmT{f'}p^{\frac p2-1}}{\nrmT fp^{\frac p2-1}}\,.
\]
\end{corollary}
A similar result also holds in the case $p>2$ and $p-1<q<p$.

\subsection{A Hamiltonian reformulation}\label{Sec:Ham}

Assume that $2<p<q$. Eq.~\eqref{ELS2} can be reformulated as a Hamiltonian system by writing $f=X$ and
\be{Eqns:Ham}\begin{array}{l}
Y=|X'|^{p-2}\,X'\;\Longleftrightarrow\;X'=|Y|^{\frac p{p-1}-2}\,Y\,,\\[4pt]
Y'=|X|^{p-2}\,X-|X|^{q-2}\,X\,.
\end{array}\ee
The energy
\[
H(X,Y)=(p-1)\,|Y|^\frac p{p-1}+\frac pq\,|X|^q-|X|^p
\]
is conserved and positive solutions are determined by the condition $\min H=\frac pq-1\le H<0$. Hence a shooting method with initial data
\[
X(0)=a\quad\mbox{and}\quad Y(0)=0
\]
provides all positive solutions (up to a translation) if $a\in(0,1]$. For clarity, we shall denote the corresponding solution by $X_a$ and $Y_a$. Some solutions of the Hamiltonian system and the corresponding vector field are shown in Fig.~\ref{F0}.

\begin{figure}[ht]
\begin{center}
\includegraphics[width=8cm,height=6cm]{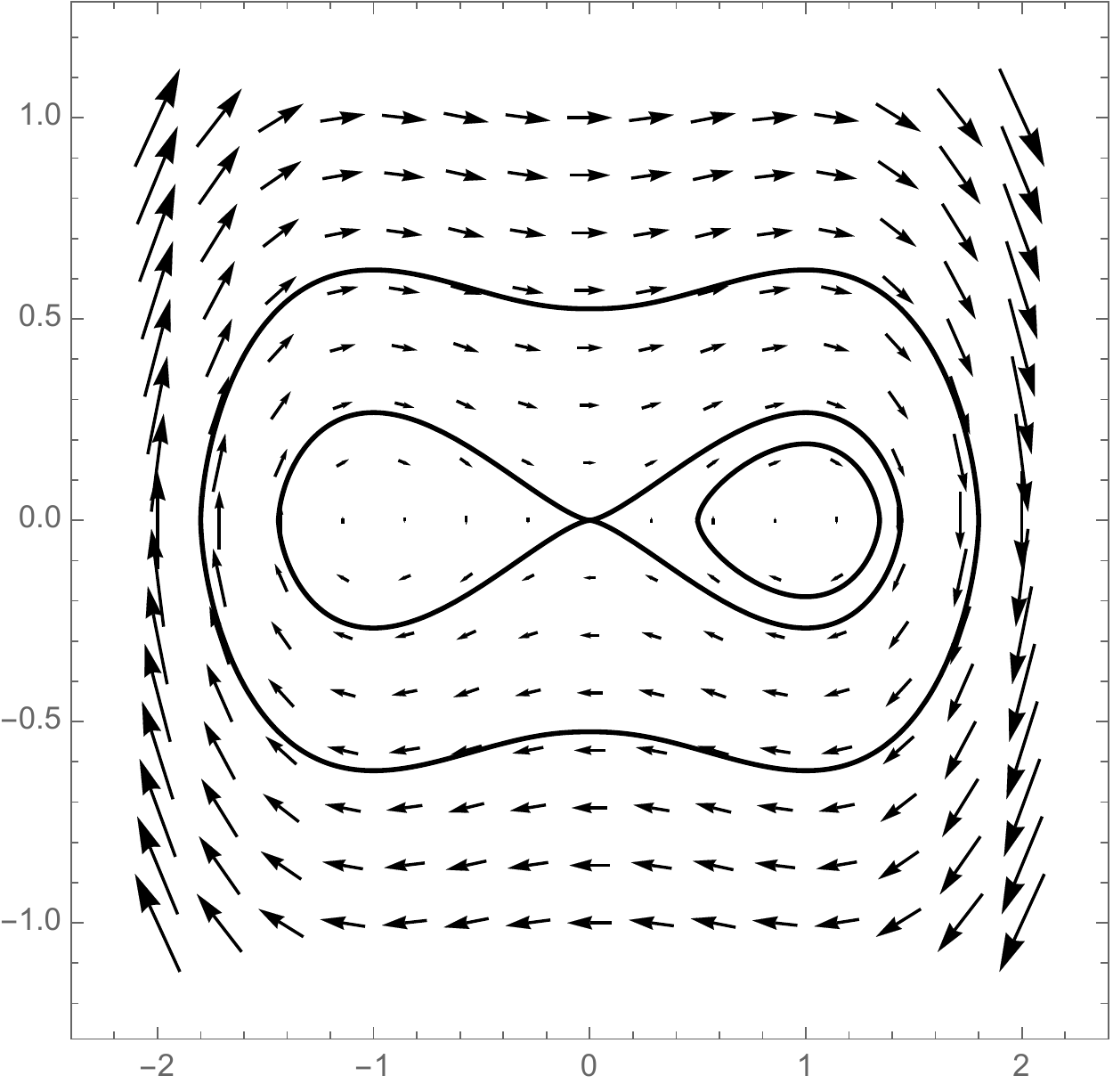}
\caption{\label{F0} The vector field $(X,Y)\mapsto(|Y|^{\frac p{p-1}-2}\,Y,|X|^{p-2}\,X-|X|^{q-2}\,X)$ and periodic trajectories corresponding to $a=1.35$ (with positive $X$) and $a=1.8$ (with sign-changing $X$) are shown for $p=2.5$ and $q=3$. The zero-energy level is also shown.}
\end{center}
\end{figure}

\subsection*{The numerical computation of the branches}

Assume that $2<p<q$ and let us denote by $f_a$ the solution~of
\[
-\,\Lp f+f^{p-1}=f^{q-1}\,,\quad f'(0)=0\,,\quad f(0)=a\,.
\]
We learn from the Hamiltonian reformulation that $H(f_a(r),f_a'(r))$ is independent of~$r$, where $H(X,Y)=(p-1)\,|Y|^\frac p{p-1}+p\,V(X)$ where $V(X):=\frac1q\,|X|^q-\frac1p\,|X|^p$. Since we are interested only in positive solutions, it is necessary that $H(a,0)<0$, which means that we can parametrize all non-constant solutions by $a\in(0,1)$. Let $b(a)\in\(1,(q/p)^{1/(q-p)}\)$ be the other positive solution of $V(b)=V(a)$.

If $T_a$ denotes the period of $f_a$, then we know that $f_a'$ is positive on the interval $(0,T_a/2)$ and can compute it using the identity $H(f_a,f_a')=V(a)$ as
\[
f_a'(r)=\(\frac p{p-1}\,\Big(V(a)-V\big(f_a(r)\big)\Big)\)^\frac1p\,.
\]
This allows to compute $T_a$ as
\[
T_a=2\int_0^{T_a/2}dr=\int_a^{b(a)}\(\frac p{p-1}\,\Big(V(a)-V\big(X\big)\Big)\)^{-\frac1p}\,dX\,.
\]
See Fig.~\ref{F1}.

\begin{figure}[ht]
\begin{center}
\includegraphics[width=8cm]{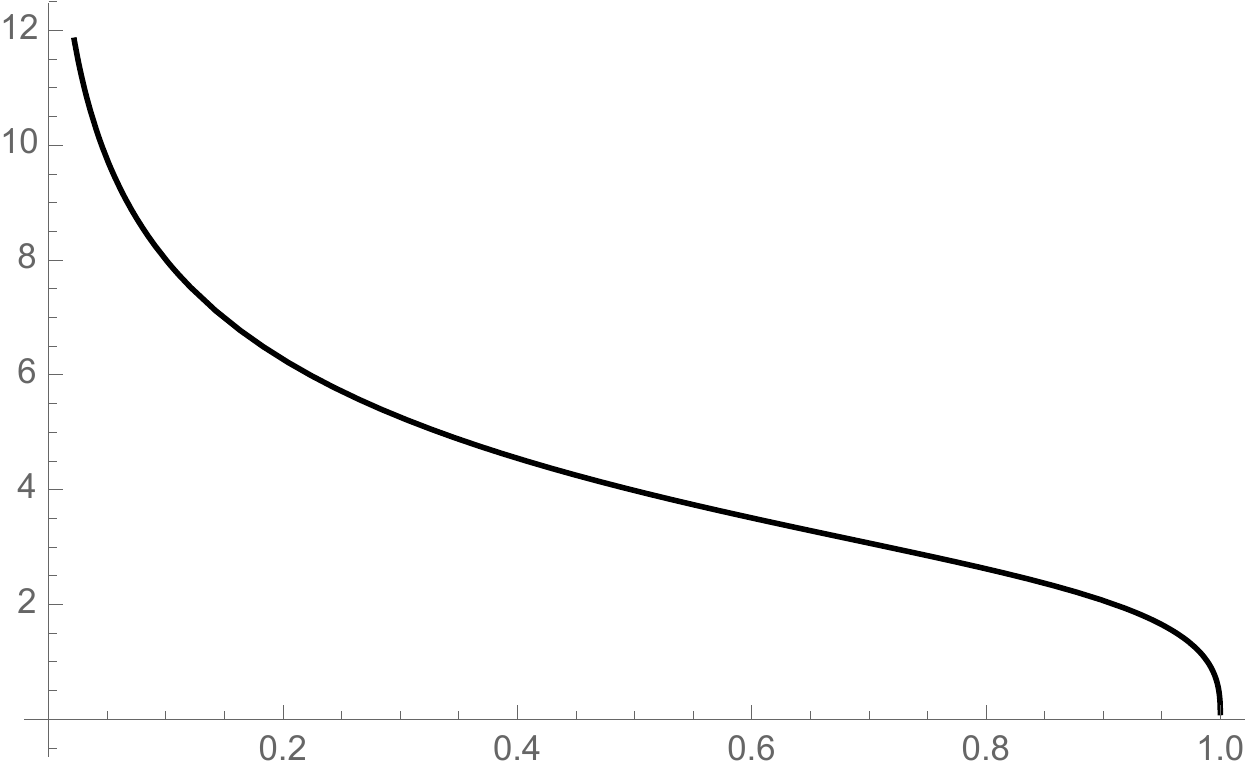}
\caption{\label{F1} The period $T_a$ of the solution of~\eqref{Eqns:Ham} with initial datum $X(0)=a\in(0,1)$ and $Y(0)=0$ as a function of $a$ for $p=3$ and $q=5$. We observe that $\lim_{a\to0}T_a=+\infty$ and $\lim_{a\to1}T_a=0$.}
\end{center}
\end{figure}

With the same change of variables $X=f_a(r)$, we can also compute
\begin{align*}
&\int_0^{T_a}|f'|^p\,dr=2\int_a^{b(a)}\(\frac p{p-1}\,\Big(V(a)-V\big(X\big)\Big)\)^{1-\frac1p}\,dX\,,\\
&\int_0^{T_a}|f|^p\,dr=\int_a^{b(a)}X^p\(\frac p{p-1}\,\Big(V(a)-V\big(X\big)\Big)\)^{-\frac1p}\,dX\,,\\
&\int_0^{T_a}|f|^q\,dr=\int_a^{b(a)}X^q\(\frac p{p-1}\,\Big(V(a)-V\big(X\big)\Big)\)^{-\frac1p}\,dX\,.
\end{align*}
Using Proposition~\ref{Prop:Equivalence}, we can obtain the plot $(\lambda,\mu(\lambda))$ as a curve parametrized by $a\in(0,1)$. See Fig.~\ref{F3}.  Similar results also holds in the case $p>2$ and $p-1<q<p$.

\begin{figure}[ht]
\begin{center}
\includegraphics[width=5cm,height=5cm]{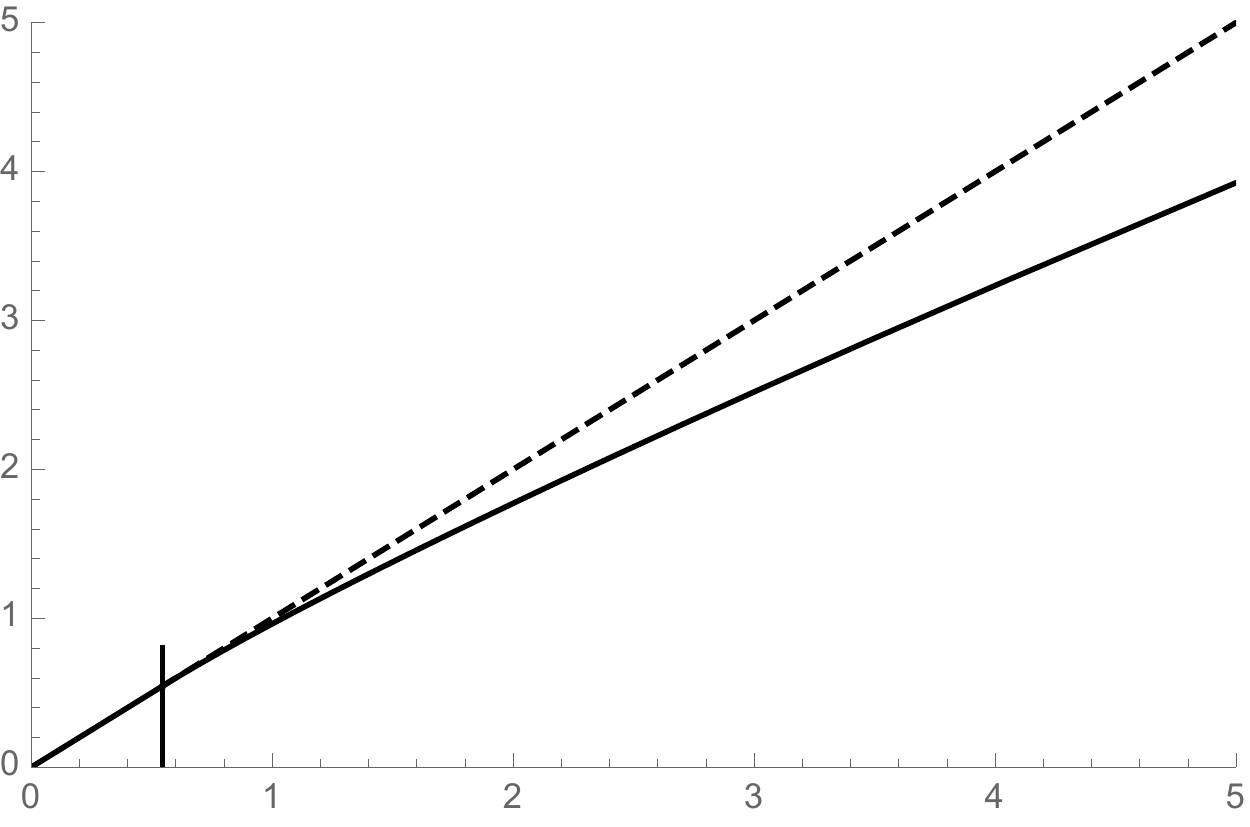}\hspace*{0.5cm}\includegraphics[width=5cm,height=5cm]{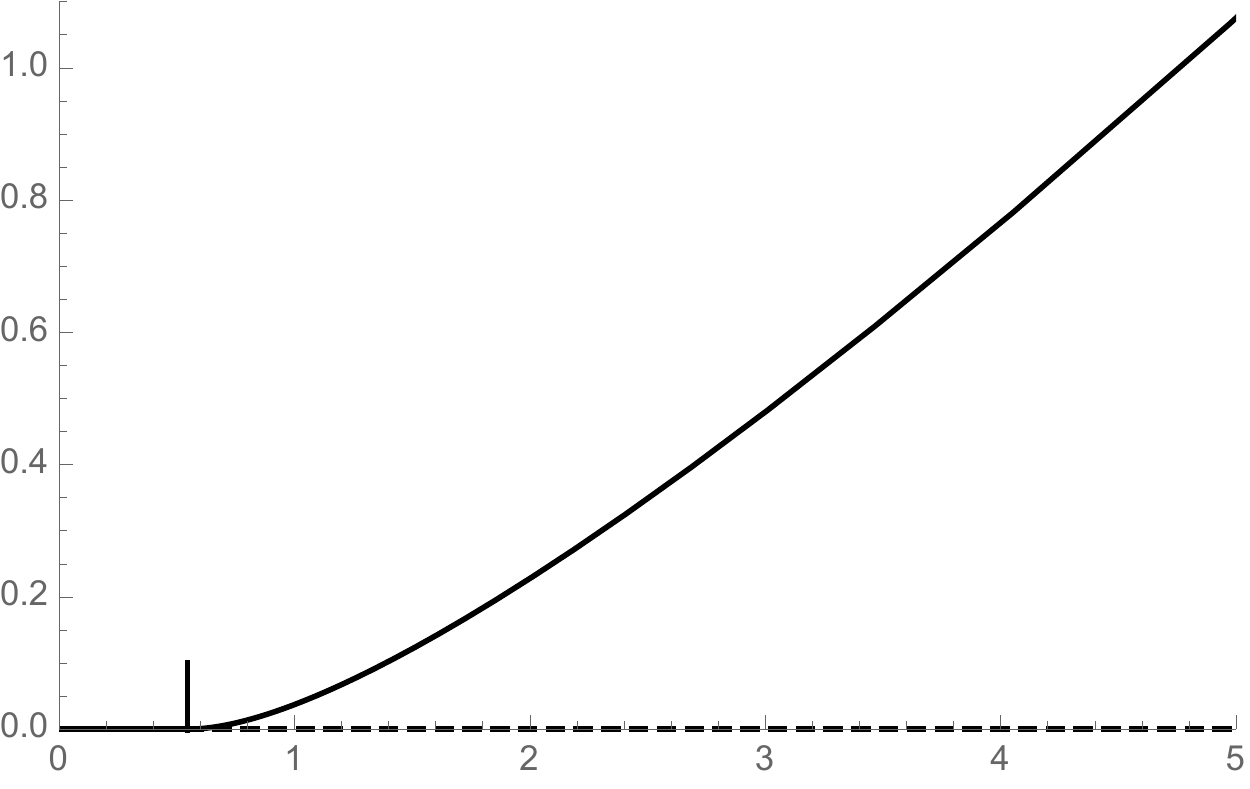}
\caption{\label{F3} Left: the branch $\lambda\mapsto\mu(\lambda)$ for $p=3$ and $q=5$. Right:the curve $\lambda\mapsto\lambda-\mu(\lambda)$. In both cases, the bifurcation point $\lambda=\lambda_1^\star$ is shown by a vertical line.}
\end{center}
\end{figure}

\section{Concluding remarks and open questions}\label{Sec:MoreOpen}

A major difference with the case $p=2$ is that the $1$-homogenous $p$-Laplacian flow~\eqref{flow} involves a nonlocal term, for homogeneity reasons. This is new and related with the fact that inequality~\eqref{Interp} is $2$-homogenous. To get rid of this constraint, one should consider inequalities with a different homogeneity, but then one would be in trouble when Taylor expanding at order two around the constants, and the framework should then be entirely different. Instead of using a $1$-homogenous flow, one could use a non $1$-homogenous flow as in~\cite{DEKL}, but one cannot expect that this would significantly remove the most important difficulty, namely that $\lambda_1^\star$ is a natural threshold for the perturbation of the constants.

In Lemma~\ref{Lem:CS}, we cannot replace $\lambda_1$ by $\lambda_1^\star$, as it is shown in the Appendix: see the discussion of the optimal constant in~\eqref{Ineq:2.3}. On the other hand, in the computation of $\mathsf i'$ in Section~\ref{Sec:Parabolic}, the term that we drop: $\iS{u^{-p}\,|u'|^{2p}}$, is definitely of lower order in the asymptotic regime as $t\to+\infty$. Actually, in~\eqref{Ineq:2.3}, if we consider $u=1+\varepsilon\,v$ and investigate the limit as $\varepsilon\to0_+$, it is clear that the inequality of Lemma~\ref{Lem:CS} degenerates into the Poincar\'e-Wirtinger inequality
\[
\iS{(\Lp v)^2}\ge\lambda_1\,\nrmS{v'}p^{2(p-1)}\quad\forall\,v\in\mathrm W^{2,p}(\S^1)\,,
\]
where $\lambda_1$ \emph{is not} the optimal constant, as can be check by writing the Euler-Lagrange equation for an optimal $w=|v'|^{p-2}\,v$. Altogether, it does not mean that one cannot prove that the optimal constant $\Lambda_{p,q}$ in Theorem~\ref{Thm:Main1} is equal to $\lambda_1^\star$ using the carr\'e du champ method, but if this can be done, it is going to be more subtle than the usual cases of application of this technique.

Finally let us point that it is a very natural and open question to ask if there is an analogue of the Poincar\'e estimate of Lemma~\ref{Lem:CS} if $p\in(1,2)$.  If yes, then we would also have an analogue of Theorem~\ref{Thm:Main1} with $1<p<2$. Notice that this issue is not covered in~\cite{MR1151712}.

\begin{center}\rule{2cm}{0.5pt}\end{center}
\appendix
\section{Considerations on some inequalities of interest}\label{Sec:Appendix:Inequalities}

We assume that $p>1$. In this appendix we collect some observations on the various inequalities which appear in this paper and how the corresponding optimal constants are related to each other.

\subsubsection*{$\bullet$ Spectral gap associated with $\Lp$} On $\S^1$, $\mathcal L_p^{-1}\kern1pt(0)$ is generated by the constant functions and there is a spectral gap, so that we have the Poincar\'e inequality
\be{Ineq1}
\iS{|u'|^p}\ge\Lambda_1\iS{|u|^p}\quad\forall\,u\in\mathrm W^{1,p}(\S^1)\quad\mbox{such that}\quad\iS{|u|^{p-2}\,u}=0\,.
\ee
The optimal constant $\Lambda_1$ is characterized by solving the Cauchy problem
\[
\Lp u+|u|^{p-2}\,u=0\quad\mbox{with}\quad u(0)=1\quad\mbox{and}\quad u'(0)=0
\]
and performing the appropriate scaling so that the solution is $2\pi$-periodic, as follows. Let $\phi_p(s)=|s|^p/p$ and denote by $p'=p/(p-1)$ the conjugate exponent. Since the ODE can be rewritten as: $\big(\phi_p'(u')\big)'+\phi_p'(u)=0$, we can introduce $v=\phi_p'(u')$ and observe that $(u,v)$ solves the system
\[
u'=\phi_{p'}'(v)\,,\quad v'=-\,\phi_p'(u)\,,\quad u(0)=1\,,\quad v(0)=0\,.
\]
The Hamiltonian energy $\phi_p(u)+\phi_{p'}(v)=1/p$ is conserved and a simple phase plane analysis shows that the solution is periodic, with a period which depends on $p$ and is sometimes denoted by $2\pi_p$ in the literature. Then the function $x\mapsto f_p(x):=u(\pi_p\,x/\pi)$ is $2\pi$-periodic and solves
\[
\Lp f_p+\Lambda_1\,|f_p|^{p-2}\,f_p=0\quad\mbox{with}\quad\Lambda_1=\(\frac\pi{\pi_p}\)^p\,.
\]
See~\cite{MR1785676,MR1735181} for more results on $\Lambda_1$ and related issues.

For any function $u\in\mathrm W^{1,p}(\S^1)$, $t\mapsto\iS{|t+u|^p}$ is a convex function which achieves its minimum at $t=0$ if $\iS{|u|^{p-2}\,u}=0$. As a consequence, we get that
\[
\Lambda_1=\min_{u\in\mathrm W^{1,p}(\S^1)}\max_{t\in\R}\frac{\nrmS{u'}p^p}{\nrmS{t+u}p^p}
\]
and
\[
\lambda_1=\inf_{v\in\mathcal W_1}\frac{\nrmS{v'}p^2}{\nrmS vp^2}\le\Lambda_1^\frac2p\,.
\]
On the other hand, since the optimal function $f_p$ is in $\mathcal W_1$, we have that
\[
\lambda_1=\Lambda_1^\frac2p\,.
\]
Alternatively, $\lambda_1$ is the optimal constant in the inequality
\be{Ineq2}
\nrmS{v'}p^2\ge\lambda_1\,\nrmS vp^2\quad\forall\,v\in\mathcal W_1\,.
\ee
Notice that the zero average condition $\iS{|u|^{p-2}\,u}=0$ in~\eqref{Ineq1} differs from the condition $\iS v=0$ in~\eqref{Ineq2}, but that the two inequalities share the same optimal functions.

\subsubsection*{$\bullet$ The inequality on $\mathrm L^2(\S^1)$} Here we consider the inequality
\be{Ineq3}
\nrmS{v'}p^2\ge\lambda_1^\star\,\nrmS v2^2\quad\forall\,v\in\mathcal W_1\,,
\ee
with optimal constant $\lambda_1^\star$. Since $d\sigma$ is a probability measure, then $\nrmS up^2-\nrmS u2^2$ has the same sign as $(p-2)$ and we have equality if and only if $u$ is constant, so that, for any $p\neq2$, we have $(p-2)\,\big(\lambda_1^\star-\lambda_1\big)\ge0$ as already noted in the introduction. If $p=2$, we have of course $\lambda_1^\star=\lambda_1$ as the two inequalities coincide. If $p\neq2$, one can characterize $\lambda_1^\star$ by solving the Cauchy problem
\[
\Lp u+u=0\quad\mbox{with}\quad u(0)=1\quad\mbox{and}\quad u'(0)=0
\]
and performing the appropriate scaling so that the solution is $2\pi$-periodic, as it has been done above for $\Lambda_1$. We can also introduce $v=\phi_p'(u')$ and observe that $(u,v)$ solves the system
\[
u'=\phi_{p'}'(v)\,,\quad v'=-\,u\,,\quad u(0)=1\,,\quad v(0)=0\,,
\]
so that trajectories differ form the ones associated $f_p$. This proves that $\lambda_1^\star\neq\lambda_1$ if $p\neq2$. See Appendix~\ref{Sec:lambda1andlambda1star} for further details.

\subsubsection*{$\bullet$ A more advanced interpolation inequality} In the proof Lemma~\ref{Lem:CS}, we establish on $\mathrm W^{1,p}(\S^1)$ the inequality
\be{Ineq4}
\nrmS{u'}p^2\,\nrmS up^{p-2}\ge\lambda_1\iS{|u|^{p-2}\,v^2}\quad\mbox{with}\quad v=u-\bar u\,,\quad\bar u=\iS u
\ee
in the case $p>2$. This inequality is optimal because equality is achieved by $u=f_p$. We can in principle consider the inequality
\[
\nrmS{u'}p^2\,\nrmS up^{p-2}-\mu_1\iS{|u|^{p-2}\,v^2}\quad\forall\,u\in\mathrm W^{1,p}(\S^1)
\]
with optimal constant $\mu_1$ and $v=u-\bar u$, for some appropriate notion of average~$\bar u$ which is not necessarily given by $\bar u=\iS u$. With the standard definition of~$\bar u$, we have shown in Lemma~\ref{Lem:CS} that $\mu_1=\lambda_1$. Any improvement on the estimate of~$\mu_1$ (with an appropriate orthogonality condition), \emph{i.e.}, a condition such that~$f_p$ is not optimal and $\mu_1>\lambda_1$, would automatically provide us with the improved estimate
\[
\Lambda_{p,q}\ge\mu_1
\]
in Theorem~\ref{Thm:Main1}. As a consequence of Theorem~\ref{Thm:Main1}, we know anyway that $\mu_1\le\lambda_1^\star$.

Inspired by the considerations on $\Lambda_1$, let us define
\[
\mu_1=\min_{u\in\mathrm W^{1,p}(\S^1)}\max_{t\in\R}\frac{\nrmS{u'}p^2\,\nrmS up^{p-2}}{\iS{|u|^{p-2}\,|u-t|^2}}\,.
\]
An elementary optimization on $t$ shows that the optimal value is $t=\bar u_p$ with
\[
\bar u_p:=\frac{\iS{|u|^{p-2}\,u}}{\iS{|u|^{p-2}}}\,.
\]
By considering again $f_p$, we see that actually $\mu_1=\lambda_1$, which proves the inequality
\be{Ineq5}
\nrmS{u'}p^2\,\nrmS up^{p-2}\ge\lambda_1\iS{|u|^{p-2}\,|u-\bar u_p|^2}\quad\forall\,u\in\mathrm W^{1,p}(\S^1)
\ee
for an arbitrary $p>2$. As a consequence of~\eqref{Ineq5}, we recover~\eqref{Ineq:2.3}. By keeping track of the equality case in the proof, we obtain that $f_p$ realizes the equality in~\eqref{Ineq:2.3}, which  proves again that the constant $\lambda_1$ is optimal in~\eqref{Ineq:2.3}.

\section{Computation of the constants \texorpdfstring{$\lambda_1$ and $\lambda_1^\star$}{lambda1 and lambda1star}}\label{Sec:lambda1andlambda1star}

Any critical point associated with $\lambda_1$ solves
\[
-\,\Lp u=\lambda_1^\frac p2\,|u|^{p-2}\,u\quad\mbox{on}\quad\S^1\approx[0,2\pi)\,.
\]
The function $f$ such that
\[
u(x)=f\(\frac{T_1\,x}{2\,\pi}\)\quad x\in[0,T_1)\,,\quad\(\frac{T_1}{2\,\pi}\)^p=\lambda_1^\frac p2
\]
is a $T_1$-periodic solution of
\[
-\,\Lp f=|f|^{p-2}\,f\,.
\]
Moreover, by homogeneity and translation invariance, we can assume that $f(0)=1$ and $f'(0)=0$. An analysis in the phase space shows that $f$ has symmetry properties and that the energy is conserved and such that $(p-1)\,|f'|^p+|f|^p=1$, so that
\[
T_1=4\int_0^1\(\frac{p-1}{1-X^p}\)^\frac1p\,dX\,.
\]
Hence we conclude that
\[\label{lambda1Num}
\lambda_1=\(\frac2\pi\int_0^1\(\frac{p-1}{1-X^p}\)^\frac1p\,dX\)^{\!2}\,.
\]

Similarly, a critical point associated with $\lambda_1^\star$ solves
\[
-\,\nrmS{u'}p^{2-p}\,\Lp u=\lambda_1^\star\,u\quad\mbox{on}\quad\S^1\approx[0,2\pi)\,.
\]
With no loss of generality, by homogeneity we can assume that $\nrmS{u'}p^{2-p}=\lambda_1^\star$ so that $u$ can be considered as $2\pi$-periodic solution on $\R$ of $-\,\Lp u=u$. By translation invariance, we can also assume that $u'(0)=0$ but the value of $u(0)=a>0$ is unknown. The function $f$ such that
\[
u(x)=a\,f\(a^{\frac2p-1}\,x\)
\]
is still a periodic solution of
\[
-\,\Lp f=f\,,
\]
with now $f(0)=1$ and $f'(0)=0$, of period
\[
T_1^\star=2\,\pi\,a^{\frac2p-1}\,.
\]
The energy $\frac 2p\,(p-1)\,|f'|^p+|f|^2=1$ is conserved, so that
\[
T_1^\star=4\int_0^1\(\frac2p\,\frac{p-1}{1-X^2}\)^\frac1p\,dX\quad\mbox{and}\quad a^{\frac2p-1}=\frac2\pi\int_0^1\(\frac2p\,\frac{p-1}{1-X^2}\)^\frac1p\,dX\,.
\]
By computing
\[
\nrmS{u'}p^p=4\,a^{3-\frac2p}\int_0^{T_1^\star/4}|f'|^p\,\frac{dx}{2\,\pi}
=\frac2\pi\,a^{3-\frac2p}\int_0^1\(\frac2p\,\frac{p-1}{1-X^2}\)^{\frac1p-1}\,dX\,,
\]
we obtain that
\[
\lambda_1^\star=\nrmS{u'}p^{2-p}=\(\frac2\pi\int_0^1\(\frac2p\,\frac{p-1}{1-X^2}\)^{\frac1p-1}\,dX\)^{\!\frac2p-1}\(\frac2\pi\int_0^1\(\frac2p\,\frac{p-1}{1-X^2}\)^\frac1p\)^{3-\frac2p}\,.
\]

\begin{figure}[ht]
\begin{center}
\includegraphics[width=8cm,height=5cm]{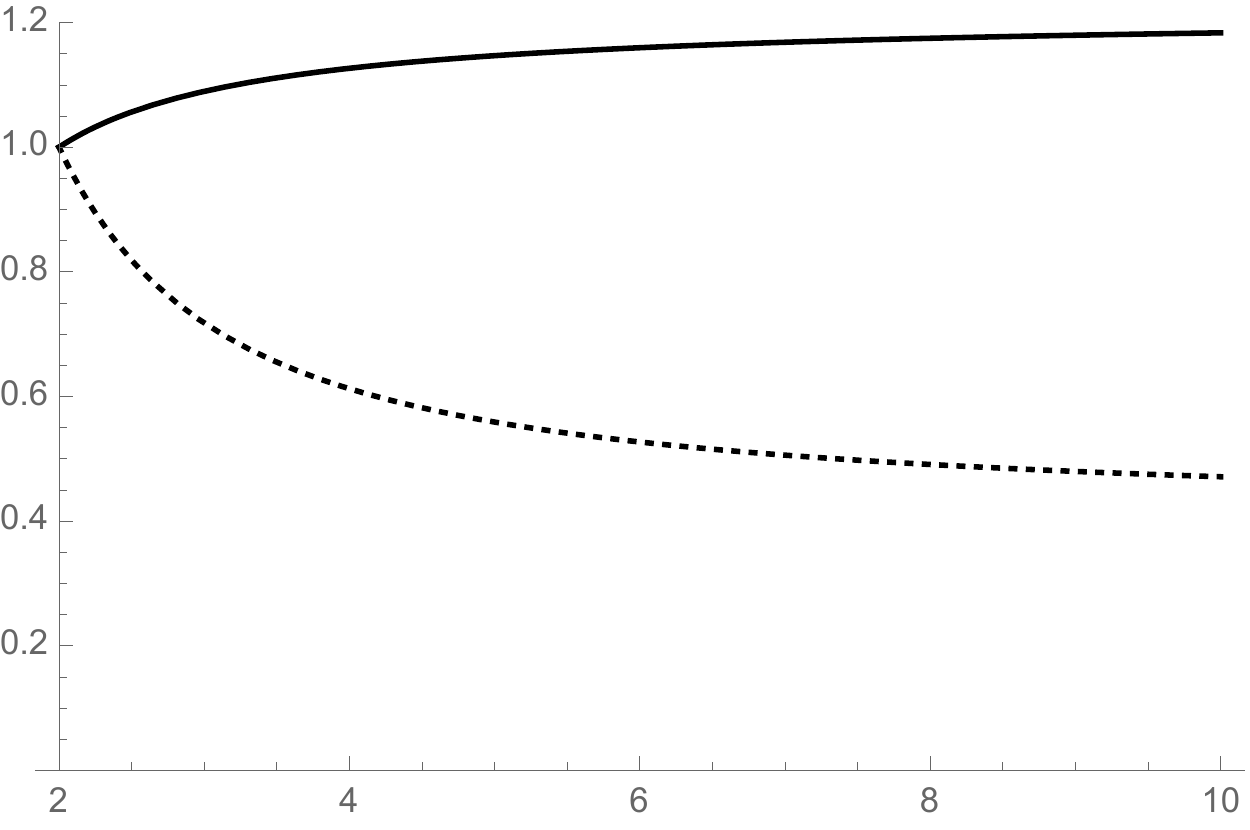}
\caption{\label{F2} The curves $p\mapsto\lambda_1$ (dotted) and $p\mapsto\lambda_1^\star$ (plain) differ.}
\end{center}
\end{figure}

\medskip\noindent{\bf Acknowledgment:} J.D.~has been partially supported by the Project EFI (ANR-17-CE40-0030) of the French National Research Agency (ANR), and also acknowledges support from the Prefalc project CFRRMA. M.G-H and R.M. have been supported by Fondecyt grant 1160540.


\begin{center}\rule{2cm}{0.5pt}\end{center}
\end{document}